\documentclass[12pt]{amsart}
\usepackage{amsmath,amssymb,amsbsy,amsfonts,latexsym,amsopn,amstext,cite,
                                               amsxtra,euscript,amscd,bm,mathabx}
\usepackage{url}
\usepackage[colorlinks,linkcolor=blue,anchorcolor=blue,citecolor=blue,backref=page]{hyperref}
\usepackage{color}
\usepackage{graphics,epsfig}
\usepackage{graphicx}
\usepackage{float} 
\usepackage[english]{babel}
\usepackage{mathtools}
\usepackage{todonotes}
\usepackage{url}
\usepackage[colorlinks,linkcolor=blue,anchorcolor=blue,citecolor=blue,backref=page]{hyperref}
\DeclareMathAlphabet{\mathmybb}{U}{bbold}{m}{n}

\usepackage[norefs,nocites]{refcheck}
\hypersetup{breaklinks=true}

\usepackage[norefs,nocites]{refcheck}
\usepackage[english]{babel}
\begin{document}

\newtheorem{thm}{Theorem}
\newtheorem{lem}[thm]{Lemma}
\newtheorem{claim}[thm]{Claim}
\newtheorem{cor}[thm]{Corollary}
\newtheorem{prop}[thm]{Proposition} 
\newtheorem{definition}[thm]{Definition}
\newtheorem{rem}[thm]{Remark} 
\newtheorem{question}[thm]{Open Question}
\newtheorem{conj}[thm]{Conjecture}
\newtheorem{prob}{Problem}
\newtheorem{Process}[thm]{Process}
\newtheorem{Computation}[thm]{Computation}
\newtheorem{Fact}[thm]{Fact}
\newtheorem{Observation}[thm]{Observation}

\newtheorem{lemma}[thm]{Lemma}

\newcommand{\GL}{\operatorname{GL}}
\newcommand{\SL}{\operatorname{SL}}
\newcommand{\lcm}{\operatorname{lcm}}
\newcommand{\ord}{\operatorname{ord}}
\newcommand{\Op}{\operatorname{Op}}
\newcommand{\Tr}{\operatorname{Tr}}
\newcommand{\Nm}{\operatorname{Nm}}
\newcommand{\BigSquare}[1]{\raisebox{-0.5ex}{\scalebox{2}{$\square$}}_{#1}}

\numberwithin{equation}{section}
\numberwithin{thm}{section}
\numberwithin{table}{section}

\numberwithin{figure}{section}

\def\sssum{\mathop{\sum\!\sum\!\sum}}
\def\ssum{\mathop{\sum\ldots \sum}}
\def\iint{\mathop{\int\ldots \int}}

\def\wt {\mathrm{wt}}
\def\Tr {\mathrm{Tr}}

\def\SrA{\cS_r\(\cA\)}

\def\vol {{\mathrm{vol\,}}}
\def\squareforqed{\hbox{\rlap{$\sqcap$}$\sqcup$}}
\def\qed{\ifmmode\squareforqed\else{\unskip\nobreak\hfil
\penalty50\hskip1em\null\nobreak\hfil\squareforqed
\parfillskip=0pt\finalhyphendemerits=0\endgraf}\fi}

\def \ss{\mathsf{s}} 

\def \balpha{\bm{\alpha}}
\def \bbeta{\bm{\beta}}
\def \bgamma{\bm{\gamma}}
\def \blambda{\bm{\lambda}}
\def \bchi{\bm{\chi}}
\def \bphi{\bm{\varphi}}
\def \bpsi{\bm{\psi}}
\def \bomega{\bm{\omega}}
\def \btheta{\bm{\vartheta}}

\newcommand{\bfxi}{{\boldsymbol{\xi}}}
\newcommand{\bfrho}{{\boldsymbol{\rho}}}

 \def \xbar{\overline x}
  \def \ybar{\overline y}

\def\cA{{\mathcal A}}
\def\cB{{\mathcal B}}
\def\cC{{\mathcal C}}
\def\cD{{\mathcal D}}
\def\cE{{\mathcal E}}
\def\cF{{\mathcal F}}
\def\cG{{\mathcal G}}
\def\cH{{\mathcal H}}
\def\cI{{\mathcal I}}
\def\cJ{{\mathcal J}}
\def\cK{{\mathcal K}}
\def\cL{{\mathcal L}}
\def\cM{{\mathcal M}}
\def\cN{{\mathcal N}}
\def\cO{{\mathcal O}}
\def\cP{{\mathcal P}}
\def\cQ{{\mathcal Q}}
\def\cR{{\mathcal R}}
\def\cS{{\mathcal S}}
\def\cT{{\mathcal T}}
\def\cU{{\mathcal U}}
\def\cV{{\mathcal V}}
\def\cW{{\mathcal W}}
\def\cX{{\mathcal X}}
\def\cY{{\mathcal Y}}
\def\cZ{{\mathcal Z}}
\def\Ker{{\mathrm{Ker}}}

\def\NmQR{N(m;Q,R)}
\def\VmQR{\cV(m;Q,R)}

\def\Xm{\cX_{p,m}}

\def \A {{\mathbb A}}
\def \B {{\mathbb A}}
\def \C {{\mathbb C}}
\def \F {{\mathbb F}}
\def \G {{\mathbb G}}
\def \L {{\mathbb L}}
\def \K {{\mathbb K}}
\def \N {{\mathbb N}}
\def \PP {{\mathbb P}}
\def \Q {{\mathbb Q}}
\def \R {{\mathbb R}}
\def \Z {{\mathbb Z}}
\def \fS{\mathfrak S}
\def \fB{\mathfrak B}

\def\Fq{\F_q}
\def\Fqr{\F_{q^r}} 
\def\ovFq{\overline{\F_q}}
\def\ovFp{\overline{\F_p}}
\def\GL{\operatorname{GL}}
\def\SL{\operatorname{SL}}
\def\PGL{\operatorname{PGL}}
\def\PSL{\operatorname{PSL}}
\def\li{\operatorname{li}}
\def\sym{\operatorname{sym}}

\def\Mob{M{\"o}bius }

\def\fF{\EuScript{F}}
\def\M{\mathsf {M}}
\def\T{\mathsf {T}}

\def\e{{\mathbf{\,e}}}
\def\ep{{\mathbf{\,e}}_p}
\def\eq{{\mathbf{\,e}}_q}

\def\\{\cr}
\def\({\left(}
\def\){\right)}

\def\<{\left(\!\!\left(}
\def\>{\right)\!\!\right)}
\def\fl#1{\left\lfloor#1\right\rfloor}
\def\rf#1{\left\lceil#1\right\rceil}

\def\Tr{{\mathrm{Tr}}}
\def\Nm{{\mathrm{Nm}}}
\def\Im{{\mathrm{Im}}}

\def \oF {\overline \F}

\newcommand{\pfrac}[2]{{\left(\frac{#1}{#2}\right)}}

\def \Prob{{\mathrm {}}}
\def\e{\mathbf{e}}
\def\ep{{\mathbf{\,e}}_p}
\def\epp{{\mathbf{\,e}}_{p^2}}
\def\em{{\mathbf{\,e}}_m}

\def\Res{\mathrm{Res}}
\def\Orb{\mathrm{Orb}}

\def\vec#1{\mathbf{#1}}
\def \va{\vec{a}}
\def \vb{\vec{b}}
\def \vh{\vec{h}}
\def \vk{\vec{k}}
\def \vs{\vec{s}}
\def \vu{\vec{u}}
\def \vv{\vec{v}}
\def \vz{\vec{z}}
\def\flp#1{{\left\langle#1\right\rangle}_p}
\def\T {\mathsf {T}}

\def\sfG {\mathsf {G}}
\def\sfK {\mathsf {K}}

\def\mand{\qquad\mbox{and}\qquad}

\title[Distribution of $\theta-$powers]
{Distribution of $\theta-$powers and their sums}

\author[Siddharth Iyer] {Siddharth Iyer}
\address{School of Mathematics and Statistics, University of New South Wales, Sydney, NSW 2052, Australia}
\email{siddharth.iyer@unsw.edu.au}

\begin{abstract}
We refine a remark of Steinerberger (2024), proving that for $\alpha \in \mathbb{R}$, there exist integers $1 \leq b_1, \dots, b_k \leq n$ such that  
\[
\left\| \sum_{j=1}^k \sqrt{b_j} - \alpha \right\| = O(n^{-\gamma_k}),
\]  
where $\gamma_k \geq (k-1)/4$, $\gamma_2 = 1$, and $\gamma_k = k/2$ for $k = 2^m - 1$. We extend this to higher-order roots.  

Building on the Bambah–Chowla theorem, we study gaps in $\{x^{\theta} + y^{\theta} : x, y \in \mathbb{N} \cup \{0\}\}$, yielding a modulo one result with $\gamma_2 = 1$ and bounded gaps for $\theta = 3/2$.  

Given $\rho(m) \geq 0$ with $\sum_{m=1}^{\infty} \rho(m)/m < \infty$, we show that the number of solutions to  
\[
\left|\sum_{j=1}^{k} a_j^{\theta} - b\right| \leq \frac{\rho\left(\|(a_1, \dots, a_k)\|_{\infty}\right)}{\|(a_1, \dots, a_k)\|_{\infty}^{k}},  
\]  
in the variables $((a_{j})_{j=1}^{k},b) \in \mathbb{N}^{k+1}$ is finite for almost all $\theta>0$. We also identify exceptional values of $\theta$, resolving a question of Dubickas (2024), by proving the existence of a transcendental $\tau$ for which $\|n^{\tau}\| \leq n^v$ has infinitely many solutions for any $v \in \mathbb{R}$. 
\end{abstract}


\maketitle

\tableofcontents
\section{Introduction}
For a real number $x$ we denote $\|x\|:= \min_{m\in \mathbb{Z}}|x-m|$. Steinerberger \cite{Steinerberger}, through the use of exponential sums, has shown that for a fixed $k \in \mathbb{N}$, there exist integers $1 \leq a_{1},\ldots,a_{k} \leq n$ such that 
\begin{align}
\label{Stein}
0<\left\|\sum_{j=1}^{k}\sqrt{a_{j}}\right\| = O\left(n^{-c \cdot k^{1/3}}\right) \text{ as }n \rightarrow \infty,
\end{align}
where $c$ is a constant independent of $k$. Dubickas \cite[Theorem 3]{Dubickas} proved a stronger and more generalised result. They show that when $d \geq 2$ is a positive integer then, there exists integers $1 \leq a_{1},\ldots,a_{k} \leq n$ such that 
\begin{align}
\label{Dub}
0<\left\| \sum_{j=1}^{k} \sqrt[d]{a_{j}}\right\| = O\left(n^{-(\frac{k-2}{d})-1}\right) \text{ as }n \rightarrow \infty.
\end{align}
We note that the result by Steinerberger also details the distribution modulo one property of sums of square roots, which is absent in \cite{Dubickas}. Following the third remark in \cite[p. 49]{Steinerberger}, with a change in terminology, they show that when $\alpha \in \mathbb{R}$, there exist integers $1 \leq a_{1},\ldots,a_{k} \leq n$ such that
\begin{align}
\label{SteinEqui}
\left\|\sum_{j=1}^{k}\sqrt{a_{j}}-\alpha\right\| = O\left(n^{-c \cdot k^{1/3}}\right) \text{ as }n \rightarrow \infty,
\end{align}
where the implied constant is independent of $\alpha$. We improve their result by replacing $O\left(n^{-c k^{1/3}}\right)$ with $O\left(n^{-(\frac{k-1}{4})}\right)$ in \eqref{SteinEqui}. More generally, we obtain the following theorem.
\begin{thm}
\label{main}
For every $k,d \in \mathbb{N}$ with $d \geq 2$, there exists an effective constant $C_{k,d} > 0$ so that for every $\alpha \in \R$ and $n \in \mathbb{N}$ there exist integers $1 \leq b_{1},\ldots,b_{k}\leq n$ with
\begin{align*}
\left\|\sum_{j=1}^{k}\sqrt[d]{b_{j}} - \alpha \right\| \leq C_{k,d}\cdot n^{-\gamma(k,d)},
\end{align*}
where
\[\gamma(k,d) = \frac{d^{\lfloor \log_{d}(k+1)\rfloor}-1}{d} \geq \frac{k-d+1}{d^2}.
\]
\end{thm}

Theorem \ref{main} establishes a bound of $O(n^{-\gamma(k,2)})$ on the largest gap in $\sum_{j=1}^{k}\sqrt{b_{j}} \mod 1$, where $b_{j} \in 1,\ldots,n$ and $n \rightarrow \infty$. This result supplements the bound of $O(n^{-2k+3/2})$ for the smallest non-zero gap, established by a direct application of Qian and Wang \cite{Qian}. Notably, their work \cite[Theorem 2]{Qian} implies the existence integers $1\leq u_{j},v_{j} \leq n$ that satisfy
\begin{align*}
0<\left|\sum_{j=1}^{k}\sqrt{u_{j}} -\sqrt{v_{j}}\right| = O(n^{-2k+3/2}).
\end{align*}
We recall for $k = 1$ and $d=2$, stronger and more subtle results on the distribution of gaps are known \cite{Gap1, Elkies}. See \cite{Cheng, Steinerberger} for further background in this study area.

Building upon work by Dubickas \cite{Dubickas}, when $\theta > 0$ and $a_{1},\ldots,a_{k}\in \N$, we study the distribution of sums $\sum_{j=1}^{k}a_{j}^{\theta} \mod 1$. Such studies yield improvements to Theorem \ref{main} by updating $\gamma(k,d)$ to $\gamma^{*}(k,d)$ where $\gamma^{*}(1,d)= 1-\frac{1}{d}$, $\gamma^{*}(2,d)= 3/2-\frac{1}{d}$ and $\gamma^{*}(k,d) = \max\{3/2-\frac{1}{d},\gamma(k,d)\}$ for $k \geq 3$.

Denoting $\{\cdot\}$ to represent the fractional part operator, the fact $\gamma^{*}(1,d)= 1-\frac{1}{d}$ follows from the following result.
\begin{lem}
When $0<\theta<1$ and $\alpha \in \mathbb{R}$, there exist an effective constant $C_{\theta} > 0$ so that
\begin{align*}
\min_{1 \leq a \leq n}\|a^{\theta}-\alpha\| \leq C_{\theta}\cdot n^{\theta-1}.
\end{align*}
\end{lem}
\begin{proof}
Let $r_{\theta}(\alpha,n)$ to be the unique real number in the interval
\begin{align*}
\left[(\lfloor n^{\theta} -1 \rfloor)^{\frac{1}{\theta}}, \lfloor n^{\theta}\rfloor^{\frac{1}{\theta}} \right),
\end{align*}
so that $\{r_{\theta}(\alpha,n)^{\theta}\} = \{\alpha\}$, and set $q_{\theta}(\alpha,n) = \lceil r_{\theta}(\alpha,n) \rceil$. Note that by the mean value theorem, we have
\begin{align*}
0 \leq q_{\theta}(\alpha,n)^{\theta} - r_{\theta}(\alpha,n)^{\theta} \leq C_{\theta}\cdot n^{\theta - 1},
\end{align*}
for some constant $C_{\theta} > 0$.
\end{proof}
The equality $\gamma^{*}(2,d)= 3/2-\frac{1}{d}$ follows as a corollary of a more general theorem.
\begin{cor}
\label{cortwotheta}
Let $\theta \in (0,1)\cup (1,3/2)$ and $\alpha \in \mathbb{R}$, there exists an effective constant $C_{\theta}>0$ so that
\begin{align*}
\min_{1 \leq u,v \leq n}\|u^{\theta} + v^{\theta} - \alpha\| \leq C_{\theta}\cdot n^{\theta - 3/2}.
\end{align*}
\end{cor}
The above corollary emerges from an investigation of gaps between consecutive elements in the set 
\begin{align*}
\BigSquare{\theta} := \{u^{\theta}+v^{\theta}: \ u,v \in \mathbb{N}\cup\{0\}\}
\end{align*}
for $\theta > 0$. The Bambah-Chowla Theorem tells us that when $x \geq 1$ there is an element of $\BigSquare{2}$ in the interval \text{$[x, x+2\sqrt{2}x^{1/4}+1)$} \cite{Shiu, Chowla}. A similar argument tells us that there exists an effective constant $C_{\theta}>0$ so that there is an element of $\BigSquare{\theta}$ in the interval $[x,x+C_{\theta}\cdot x^{1-\frac{2}{\theta}+\frac{1}{\theta^2}}]$ for $\theta > 1$. We can improve over such an argument in the $1 < \theta < 2$ range.
\begin{thm}
\label{main5}
For $\theta > 0$ and $x \geq 1$ there exists a constant $D_{\theta} > 0$ so that there is an element of $\BigSquare{\theta}$ in the interval $[x,x+D_{\theta}\cdot x^{\Psi(\theta)}]$, where
\[ 
\Psi(\theta) = \begin{cases}
            1-\frac{3}{2\theta} & 0<\theta<1\\
            0 & \theta =1\\
            1-\frac{3}{2\theta} & 1 < \theta < 2\\
            1-\frac{2}{\theta}+\frac{1}{\theta^2} & \theta \geq 2.
            \end{cases}
\]
\end{thm}
Steinerberger also establishes a lower bound on the gap between sums of square roots and integers. They show that if $\sum_{j=1}^{k}\sqrt{a_{j}}$ is not an integer then there exists an absolute constant $C>0$ such that
\begin{align*}
\left\|\sum_{j=1}^{k}\sqrt{a_{j}}\right\| \geq C\cdot n^{1/2-2^{k-1}}
\end{align*}
in \cite[p. 49]{Steinerberger}. Dubickas generalises this result for higher order roots in \cite[Theorem 1]{Dubickas}. They show that if $\sum_{j=1}^{k}\sqrt[d]{a_{j}}$ is not an integer then
\begin{align*}
\left\|\sum_{j=1}^{k}\sqrt[d]{a_{j}}\right\| > (2k+1)^{1-d^{k}}\cdot n^{1/d - d^{k-1}}.
\end{align*}
As a consequence of Theorem \ref{almost}, we show that for almost all $\theta>0$ (in the Lebesgue measure sense), if $\varepsilon>0$, there exists a constant $C_{\theta, \varepsilon}>0$ such that if $\sum_{j=1}^{k}a_{j}^{\theta}$ is not an integer then
\begin{align*}
\left\|\sum_{j=1}^{k}a_{j}^{\theta}\right\| \geq C_{\theta,\varepsilon}\cdot n^{-k-\varepsilon}.
\end{align*}
\begin{thm}
\label{almost}
Given $\rho(\cdot) \geq 0$ with $\sum_{m=1}^{\infty} \rho(m)/m < \infty$, and $k \in \mathbb{N}$, the number of solutions to  
\[
\left|\sum_{j=1}^{k} a_j^{\theta} - b\right| \leq \frac{\rho\left(\|(a_1, \dots, a_k)\|_{\infty}\right)}{\|(a_1, \dots, a_k)\|_{\infty}^{k}},  
\]  
in the variables $\left((a_{j})_{j=1}^{k},b\right) \in \mathbb{N}^{k+1}$ is finite for almost all $\theta>0$.
\end{thm}
In \cite[Theorem 5]{Dubickas}, it is stated that, assuming the abc-conjecture, the inequality $0 < \|n^{\theta}\| \leq n^{-1-\varepsilon}$ has finitely many solutions for $n \in \mathbb{N}$ when $\theta < 1$ is a fixed rational number. Theorem \ref{almost} extends this result, showing that the same holds for almost all $\theta > 0$.

The exceptional set of $\theta$ in Theorem \ref{almost} is related to a question posed by Dubickas. Specifically (with a change in terminology), they ask whether, for a given transcendental number $\tau$ and $\varepsilon > 0$, there exists a smallest $v(\tau) \in \mathbb{R}$ such that 
\begin{align*}
0<\|n^{\tau}\| \leq n^{v(\tau)-\varepsilon}
\end{align*}
has only finitely many solutions for $n \in \mathbb{N}$ \cite[p. 5]{Dubickas}. We prove that the answer to this question is ``no''.
\begin{thm}
\label{DubCounter}
Let $\Phi : \mathbb{N} \rightarrow \R_{>0}$ be a function. The collection of $\theta \in \R_{>0}$ so that there are infinitely many solutions to $0<\|n^{\theta}\| \leq \Phi(n)$ forms a dense uncountable set.
\end{thm}
As a corollary, if we take $\Phi(n) = 2^{-n}$ in Theorem \ref{DubCounter} we discover the existence of uncountably many transcendental numbers $\tau > 0$ so that $0 < \|n^{\tau}\| \leq 2^{-n}$ has infinitely many solutions in $n \in \mathbb{N}$. These transcendental numbers qualify for exceptional values of $\theta$ in Theorem \ref{almost}.

It remains an open question whether there exists an analogue of the Duffin–Schaeffer theorem \cite{Duffin, Maynard} for Theorem \ref{almost}.
\begin{conj}
Given $\rho(\cdot) \geq 0$ with $\sum_{m=1}^{\infty} \rho(m)/m = \infty$, and $k \in \mathbb{N}$, the number of solutions to  
\[
\left|\sum_{j=1}^{k} a_j^{\theta} - b\right| \leq \frac{\rho\left(\|(a_1, \dots, a_k)\|_{\infty}\right)}{\|(a_1, \dots, a_k)\|_{\infty}^{k}},  
\]  
in the variables $\left((a_{j})_{j=1}^{k},b\right) \in \mathbb{N}^{k+1}$ is infinite for almost all $\theta>0$.
\end{conj}

\section{Notation}
\label{Notation}
Let $\xi,d \in \mathbb{N} = \{1,2,\ldots\}$ with $d \geq 2$.
\begin{itemize}
\item $\zeta_{d}$ is the primitive root of unity $e^{\frac{2\pi i }{d}}$, dividing the $d-$th cyclotomic polynomial $\Psi_{d}$ of degree $\varphi(d)$, where $\varphi$ is the Euler totient function.
\item $p_{j}$ is the $j-$th prime, we let $p_{0} = 1$ and $p_{-1} = 0$.
\item $P_{\xi} = \{\prod_{j=1}^{\xi}p_{j}^{a_{j}}:a_{j}\in \{0,1,\ldots,d-1\}\},$ so that $P_{\xi}$ is the set of $d^{\xi}$ natural numbers dividing the product $p_{1}^{d-1}\ldots p_{\xi}^{d-1}$,
\item $P_{\xi}^{*} := P_{\xi}\setminus\{1\}$,
\item $S_{\xi} := \{\sum_{f \in P_{\xi}^{*}}c_{f}\sqrt[d]{f}: c_{f}\in \mathbb{Z}\}$,
\item $S_{\xi}^{*} = S_{\xi}\setminus \{0\}$.
\item If $x \in \mathbb{R}$ we denote $\lfloor x \rfloor$ , $\lceil x\rceil$ and $\{x\}$ to be the floor, ceiling and integer part of $x$, respectively.
\item Let $U$, $V$, and $W$ be real quantities dependent on parameters $\mathbf{P}$ from some set, along with (possibly) $n \in \mathbb{N}, x \in \mathbb{R}$ and $\alpha \in \mathbb{R}$. We write $U \ll V$, $V \gg U$ or $U = O(V)$ if there exists an effective constant $C > 0$, independent of $n,x$ and $\alpha$, such that $|U| \leq C\cdot V$ for all choices of $n,x$ and $\alpha$ (with other variables fixed). Further, we write $U = V+O(W)$ if $U-V \ll W$.
\item $\mu$ is the Lebesgue measure defined on subsets of $\mathbb{R}$.
\item If $\mathbf{x} = (x_{1},\ldots,x_{k}) \in \R^{k}$ we denote $\| \mathbf{x}\|_{\infty} = \max_{1 \leq j \leq k} |x_{j}|$.
\end{itemize}
\section{Proof of Theorem \ref{main}}
\begin{lem}
\label{Galois}
$\Q(\zeta_{d},\sqrt[d]{p_{1}},\ldots,\sqrt[d]{p_{\xi}}) / \mathbb{Q}$ is a Galois extension.
\end{lem}
\begin{proof}
Note that $\Q(\zeta_{d}) / \Q$ is Galois. Let $t \in \N$ and note that the minimal polynomial $P_{t}(x)$ for $\sqrt[d]{p_{t}}$ over $\mathbb{Q}(\zeta_{d},\sqrt[d]{p_{0}},\ldots,\sqrt[d]{p_{t-1}})$ divides $x^{d}-p_{t}$, a separable polynomial that splits in $\Q(\zeta_{d},\sqrt[d]{p_{0}},\ldots,\sqrt[d]{p_{t}})$. If $\mathbb{Q}(\zeta_{d},\sqrt[d]{p_{0}},\ldots,\sqrt[d]{p_{t-1}})\subseteq F\subseteq \mathbb{Q}(\zeta_{d},\sqrt[d]{p_{0}},\ldots,\sqrt[d]{p_{t}})$ is a field for which $P_{t}(x)$ splits, then it contains $\sqrt[d]{p_{t}}$ and hence $F= \mathbb{Q}(\zeta_{d},\sqrt[d]{p_{0}},\ldots,\sqrt[d]{p_{t}})$. Thus 
\begin{align*}
\Q(\zeta_{d},\sqrt[d]{p_{0}},\ldots,\sqrt[d]{p_{t}}) / \mathbb{Q}(\zeta_{d},\sqrt[d]{p_{0}},\ldots,\sqrt[d]{p_{t-1}})
\end{align*}
is Galois. We now observe that since we are dealing with finite extensions, by constructing automorphisms in a natural way one has
\begin{align*}
|\text{Aut}(\Q(\zeta_{d},&\sqrt[d]{p_{1}},\ldots,\sqrt[d]{p_{\xi}}) / \mathbb{Q})| \\
&\geq\prod_{t=0}^{\xi}\left|\text{Aut}(\Q(\zeta_{d},\sqrt[d]{p_{0}},\ldots,\sqrt[d]{p_{t}}) / \mathbb{Q}(\zeta_{d},\sqrt[d]{p_{0}},\ldots,\sqrt[d]{p_{t-1}}))\right|\\
&=\prod_{t=0}^{\xi}\left[\Q(\zeta_{d},\sqrt[d]{p_{0}},\ldots,\sqrt[d]{p_{t}}) : \mathbb{Q}(\zeta_{d},\sqrt[d]{p_{0}},\ldots,\sqrt[d]{p_{t-1}}))\right]\\
&=[\Q(\zeta_{d},\sqrt[d]{p_{1}},\ldots,\sqrt[d]{p_{\xi}}) : \mathbb{Q}].
\end{align*}
Hence the result follows.
\end{proof}

For an element $w \in S_{\xi}$ of the form $\sum_{f \in P_{\xi}^{*}}c_{f}\sqrt[d]{f}$ we denote
\begin{align*}
M(w) := \max_{f \in P_{\xi}^{*}}|c_{f}|.
\end{align*}

\begin{lem}
\label{LowerBound}
Let $w \in S_{\xi}^{*}$ with $M(w) = n$. We have
\begin{align*}
\|w\| \gg n^{-d^{\xi}+1}.
\end{align*}
\end{lem}
\begin{proof}
Let $g_{w}$ be an integer so that $\|w\| = |w-g_{w}|$ and put $w^{*}=w-g_{w}$. Observe the bound $g_{w} \ll n$. Since $w^{*}$ is an algebraic integer the minimal polynomial for $w^{*}$ over $\Q$ is monic. Let this polynomial be denoted as $Q_{w}(x)$. By \cite[Theorem 2]{Besico} we have $w^{*} \neq 0$ and thus $\text{deg}(Q_{w}) \geq 2$. The same theorem tells us that $\text{deg}(Q_{w}) \leq d^{\xi}$. Due to  Lemma \ref{Galois} if $v$ is a root of $Q_{w}$ then there exists $\sigma \in \text{Aut}(\Q(\zeta_{d},\sqrt[d]{p_{1}},\ldots,\sqrt[d]{p_{\xi}}) / \mathbb{Q})$ so that $v =\sigma(w^{*})$. We now observe that if $\sigma \in \text{Aut}(\Q(\zeta_{d},\sqrt[d]{p_{1}},\ldots,\sqrt[d]{p_{\xi}}) / \mathbb{Q})$ then there exists integers $j_{1},\ldots,j_{\xi}$ so that $\sigma(\sqrt[d]{p_{t}}) = \zeta_{d}^{j_{t}}\sqrt[d]{p_{t}}$. Writing $w^{*} = \sum_{f \in P_{\xi}^{*}}c_{f}\sqrt[d]{f} - g_{w}$ we observe that $\sigma(w^{*}) \ll n$. Let $w^{*}, w_{1},\ldots,w_{s-1}$ be all the roots of $Q_{w}$. We have
\begin{align*}
\left|w^{*}w_{1}\ldots w_{s-1}\right| = |Q_{w}(0)| \geq 1,
\end{align*}
and thus $|w^{*}| \gg n^{-s+1} \gg n^{-d^{\xi}+1}$.
\end{proof}

\begin{lem}
\label{UpperBound}
For every $n \in \mathbb{N}$ there exists $w \in S_{\xi}^{*}$ with $M(w) \leq n$ so that
\begin{align*}
\|w\|\leq n^{-d^{\xi}+1}.
\end{align*}
\end{lem}
\begin{proof}
Let
\begin{align*}
A_{n} = \left\{\sum_{f \in P_{\xi}^{*}}d_{f}\sqrt[d]{f}: d_{f}\in \{1,\ldots,n\}\right\}.
\end{align*}
By Lemma \cite[Theorem 2]{Besico} we have $\# A_{n} = n^{d^{\xi}-1}$. After applying the Dirichlet principle to the unit torus $\mathbb{R}/\mathbb{Z}$, the proposition now follows.
\end{proof}

\begin{lem}
\label{AlgorithmicIdea}
 Whenever $\alpha \in [0,1)$ and $n \in \mathbb{N}$, there exists $w_{n} \in S_{\xi}$ satisfying $M(w_{n}) \leq n$ with
\begin{align*}
0 \leq \alpha - \{w_{n}\} \ll n^{-d^{\xi}+1}.
\end{align*}
\end{lem}
\begin{proof}
Let $j \geq 1$ be an integer, using Lemma \ref{UpperBound} and multiplication by $-1$ if necessary we can find an element $x_{j} \in S_{\xi}^{*}$ which satisfies $M(x_{j}) \leq 2^{j}$ and 
\begin{align*}
0<\{x_{j}\}\leq \frac{1}{(2^{j})^{d^{\xi}-1}}.
\end{align*}
By Lemma \ref{LowerBound}, we may strengthen the above inequality to
\begin{align*}
\frac{A_{\xi}}{(2^{j})^{d^{\xi}-1}}\leq \{x_{j}\}\leq \frac{1}{(2^{j})^{d^{\xi}-1}}.
\end{align*}
for some $A_{\xi} > 0$.
Form the sequences $(y_{h})_{h=1}^{\infty}$ and $(\alpha_{h})_{h=1}^{\infty}$ inductively as follows; choose $y_{1}$ to be the largest non-negative integer so that 
\begin{align*}
\alpha - y_{1}\{x_{1}\} \geq 0 
\end{align*}
and set $\alpha_{1} = \alpha -y_{1}\{x_{1}\}$. For convenience, we estimate that
\begin{align*}
y_{1} \leq \beta,
\end{align*}
where
\begin{align*}
\beta = \left\lceil \frac{2^{d^{\xi}-1}}{A_{\xi}}\right\rceil.
\end{align*}
Furthermore, we estimate that 
\begin{align*}
0\leq \alpha_{1} \leq \frac{1}{2^{d^{\xi}-1}}.
\end{align*}
Let $y_{h+1}$ be the largest non-negative integer so that 
\begin{align*}
\alpha_{h}-y_{h+1}\{x_{h+1}\} \geq 0.
\end{align*}
Then we set
\begin{align*}
\alpha_{h+1} = \alpha_{h}-y_{h+1}\{x_{h+1}\}.
\end{align*}
In this setting, we observe that
\begin{align*}
0\leq \alpha_{h} \leq \frac{1}{(2^{h})^{d^{\xi}-1}}.
\end{align*}
With $\alpha_{0} = \alpha$ and $h \in 0,1,2,\ldots$ we have
\begin{align*}
0 \leq \alpha_{h+1} = \alpha_{h}-y_{h+1}\{x_{h+1}\} \leq \frac{1}{(2^{h})^{d^{\xi}-1}}-y_{h+1}\frac{A_{\xi}}{(2^{h+1})^{d^{\xi}-1}},
\end{align*}
thus
\begin{align*}
y_{h+1} \leq \beta.
\end{align*}
Now we observe that 
\begin{align}
\label{MajorInequality}
0\leq \alpha_{t} = \alpha - \sum_{h=1}^{t}y_{h}\{x_{h}\} \leq \frac{1}{(2^{t})^{d^{\xi}-1}}.
\end{align}
Put $\omega(t) = \sum_{h=1}^{t}y_{h}x_{h}$. Since $\sum_{h=1}^{t}y_{h}\{x_{h}\} \in [0,1)$ it holds true that
\begin{align*}
\sum_{h=1}^{t}y_{h}\{x_{h}\} = \left\{\sum_{h=1}^{t}y_{h}\{x_{h}\}\right\} = \left\{\sum_{h=1}^{t}y_{h}x_{h} \right\} = \left\{\omega(t)\right\}.
\end{align*}
Thus, by \eqref{MajorInequality} we have
\begin{align}
\label{MajorInequality2}
0 \leq \alpha -\{\omega(t)\} \leq \frac{1}{(2^t)^{d^{\xi}-1}}.
\end{align}
Here we bound 
\begin{align}
\label{MajorInequality3}
M(\omega(t)) = M\left(\sum_{h=1}^{t}y_{h}x_{h}\right) \leq \sum_{h=1}^{t}\beta\cdot 2^{h} \leq \beta \cdot 2^{t+1}.
\end{align}
If $n \geq 4\beta$ we can show that $\omega\left(\left\lfloor\log_{2}(\frac{n}{2\beta})\right\rfloor\right)$ satisfies
\begin{align*}
M\left(\omega\left(\left\lfloor\log_{2}\left(\frac{n}{2\beta}\right)\right\rfloor\right)\right) \leq n
\end{align*}
by \eqref{MajorInequality3}, and
\begin{align*}
0 \leq \alpha - \left\{\omega\left(\left\lfloor\log_{2}(\frac{n}{2\beta})\right\rfloor\right)\right\} &\leq \left(2^{\left\lfloor\log_{2}(\frac{n}{2\beta})\right\rfloor}\right)^{-d^{\xi}+1}\\
&\hspace{1cm}\ll \left(2^{\log_{2}(\frac{n}{2\beta})}\right)^{-d^{\xi}+1} \ll n^{-d^{\xi}+1}
\end{align*}
by \eqref{MajorInequality2}.
\end{proof}

\begin{lem}
\label{LastOne}
For every $n \in \mathbb{N}$ and $\alpha \in \R$ there exist integers \[0<c_{f}\leq n,\] so that
\begin{align*}
\left\|\sum_{f \in P_{\xi}^{*}}c_{f}\sqrt[d]{f} -\alpha\right\| \ll n^{-d^{\xi}+1}.
\end{align*}
\end{lem}
\begin{proof}
We can assume that $n \geq 6$. Put 
\begin{align*}
\rho = \sum_{f \in P_{\xi}^{*}}\left\lfloor\frac{n}{2}\right\rfloor\sqrt{f}
\end{align*}
and $\alpha_{0} = \{\alpha -\rho\}$. By Lemma \ref{AlgorithmicIdea} there exists $\kappa \in S_{\xi}$ so that \text{ }\text{ } $M(\kappa) \leq \lfloor n/3\rfloor$ and 
\begin{align*}
0 \leq \alpha_{0} - \{\kappa\} \ll n^{-d^{\xi}+1}.
\end{align*}
In particular $\kappa+\rho$ is of the form $\sum_{f \in P_{\xi}^{*}}d_{f}\sqrt{f}$ where 
\begin{align*}
\left\lfloor\frac{n}{2}\right\rfloor - \left\lfloor\frac{n}{3}\right\rfloor\leq d_{f} \leq \left\lfloor\frac{n}{2}\right\rfloor+\left\lfloor\frac{n}{3}\right\rfloor,
\end{align*}
and
\begin{align*}
\|\kappa + \rho - \alpha\| \ll n^{-d^{\xi}+1}.
\end{align*}
\par \vspace{-\baselineskip} \qedhere
\end{proof}
We are finally in a position to prove Theorem \ref{main}. Let
\begin{align*}
\xi = \lfloor \log_{d}(k+1) \rfloor.
\end{align*}
We can assume that
\begin{align*}
n \geq 2^{d} \prod_{j=1}^{\xi}p_{j}^{d-1}.
\end{align*}
Using Lemma \ref{LastOne} we can find integers 
\begin{align*}
0<c_{f} \leq \left\lfloor\frac{\sqrt[d]{n}}{\sqrt[d]{p_{1}^{d-1}\ldots p_{\xi}^{d-1}}}\right\rfloor
\end{align*}
so that 
\begin{align*}
\left\|\sum_{f \in P_{\xi}^{*}}c_{f}\sqrt[d]{f} -\alpha\right\| \ll \left\lfloor\frac{\sqrt[d]{n}}{\sqrt[d]{p_{1}^{d-1}\ldots p_{\xi}^{d-1}}}\right\rfloor^{-d^{\xi}+1} \ll n^{\frac{-d^{\xi}+1}{d}}
\end{align*}
Let $1=b_{1}=\ldots=b_{k+1-d^{\xi}}$. Then
\begin{align*}
\left\|\sum_{f \in P_{\xi}^{*}}\sqrt[d]{c_{f}^d f} + \sum_{d=1}^{k+1-d^{\xi}}\sqrt{b_{d}} -\alpha\right\|\ll n^{\frac{-d^{\xi}+1}{d}},
\end{align*}
which concludes the proof of Theorem \ref{main}.
\section{Proof of Theorem \ref{main5} and Corollary \ref{cortwotheta}}
\hspace{-0.4cm}If $\theta = 1$ the proposition holds trivially. We consider the cases $\theta \in (0,1)\cup (1,2)$ and $\theta \geq 2$ separately.
\subsection{The case $\theta \in (0,1)\cup (1,2)$} \text{ }

\text{ }

 \hspace{-0.4cm}Given $x \geq 2$, for $\theta \in (0,1)$ we let $s_{\theta}(x) = \lceil(x/2)^{1/\theta}\rceil$, and for $\theta \in (1,2)$ we let $s_{\theta}(x) = \lfloor (x/2)^{1/\theta} \rfloor$. Let $E_{\theta}(x) = x-2\cdot s_{\theta}(x)^{\theta}$. We can easily verify the following lemma.
\begin{lem}
\label{Errorsignlem}
For $\theta \in (0,1)$ we have $E_{\theta}(x) \leq 0$ and for $\theta \in (1,2)$ we have $E_{\theta}(x) \geq 0$.
\end{lem}
\begin{lem}
\label{Errorsizelem} $E_{\theta}(x) \ll x^{1-1/\theta}$.
\end{lem}
\begin{proof}
Write $E_{\theta}(x) = 2\left((x/2)^{1/\theta}\right)^{\theta} - 2s_{\theta}(x)^{\theta}$. By the mean value theorem, we have 
\begin{align*}
\left| E_{\theta}(x)\right| = 2\theta \cdot c_{\theta}(x)^{\theta - 1}\left|(x/2)^{1/\theta} -s_{\theta}(x) \right|,
\end{align*}
where $c_{\theta}(x)$ is a quantity in the bounded closed interval with endpoints being $(x/2)^{1/\theta}$ and $s_{\theta}(x)$.  Hence $E_{\theta}(x) \ll \left(x^{1/\theta}\right)^{\theta - 1} = x^{1-\frac{1}{\theta}}$.
\end{proof}
For $Q > 0$ let $H_{\theta,Q}(t) := (Q+t)^{\theta} + (Q-t)^{\theta}$.
\begin{lem}
\label{incdeclem}
As a function of $t$, $H_{\theta,Q}(t)$ is strictly decreasing on the interval $[0,Q]$ when $\theta \in (0,1)$, conversely it is strictly increasing on the same interval when $\theta \in (1,2)$.
\end{lem}
\begin{proof}
Suppose $t \in (0, Q)$, we differentiate to obtain
\begin{align*}
H_{\theta,Q}'(t) = \theta\left((Q+t)^{\theta-1} - (Q-t)^{\theta-1}\right).
\end{align*}
The above quantity is negative when $\theta \in (0,1)$ and positive when $\theta \in (1,2)$, finishing the proof.
\end{proof}
Consider the function $I_{\theta,x}(t) := (s_{\theta}(x)+t)^{\theta}+(s_{\theta}(x)-t)^{\theta}$ over the domain $[0,s_{\theta}(x)]$.
\begin{lem}
\label{kthetasize1}
For all sufficiently large $x \geq N_{\theta}$ there exists a unique real number $k_{\theta}(x) \in [0, s_{\theta}(x)]$ which satisfies $I_{\theta ,x}(k_{\theta}(x)) = x$. Furthermore $k_{\theta}(x) = o(s_{\theta}(x))$.
\end{lem}
\begin{proof}
Let $\varepsilon \in (0,1)$. We evaluate
\begin{align*}
\left|I_{\theta,x}(\varepsilon \cdot s_{\theta}(x)) - I_{\theta,x}(0)\right| &= s_{\theta}(x)^{\theta}\left| (1+\varepsilon)^{\theta}+(1-\varepsilon)^{\theta}-2\right|\\
&= s_{\theta}(x)^{\theta}\left|H_{\theta,1}(\varepsilon)-2\right|\\
&\gg x.
\end{align*}
The above computation combined with Lemmas \ref{Errorsignlem}, \ref{Errorsizelem}, \ref{incdeclem} shows that for sufficiently large $x$ there exists $k_{\theta}(x) \in [0,\varepsilon \cdot s_{\theta}(x)]$ satisfying
\begin{align*}
I_{\theta,x}(k_{\theta}(x)) - I_{\theta,x}(0) = E_{\theta}(x).
\end{align*}
The uniqueness of $k_{\theta}(x)$ follows from Lemma \ref{incdeclem}.
\end{proof}
Consider the function $G_{\theta}(z):= (1+z)^{\theta} + (1-z)^{\theta}$ over the domain $[0,1]$.
\begin{lem}
\label{Glocbound}
For $\varepsilon \in (0,\frac{1}{2}]$ we have
\begin{align*}
\left|G_{\theta}(z)-2-\theta(\theta-1)z^2\right| \leq 64\varepsilon \cdot z^2,
\end{align*}
and
\begin{align*}
\left|G_{\theta}'(z)\right| \leq (2\theta\left|\theta-1\right| + 96\varepsilon)\cdot z
\end{align*}
in the interval $z \in [0,\varepsilon]$.
\end{lem}
\begin{proof}
We compute
\begin{align*}
G_{\theta}'(z) = \theta(1+z)^{\theta-1}-\theta(1-z)^{\theta-1},
\end{align*}
\begin{align*}
G_{\theta}''(z) = \theta(\theta - 1)(1+z)^{\theta-2}+\theta(\theta-1)(1-z)^{\theta-2},
\end{align*}
and
\begin{align*}
G_{\theta}'''(z) = \theta(\theta - 1)(\theta-2)(1+z)^{\theta-3}-\theta(\theta-1)(\theta-2)(1-z)^{\theta-3}.
\end{align*}
In the interval $z\in [0,1/2]$ we have the bound
\begin{align*}
|G_{\theta}'''(z)| \leq 2\cdot\theta(\theta + 1)(\theta+2)2^{3-\theta} \leq 384.
\end{align*}
By the extended mean value theorem, we know that
\begin{align*}
G_{\theta}(z) = G_{\theta}(0) + G_{\theta}'(0)z + \frac{G_{\theta}''(0)z^2}{2!} + \frac{G_{\theta}'''(c_{z}) z^3}{3!}
\end{align*}
for some $c_{z} \in [0,z]$, and thus we have
\begin{align*}
\left|G_{\theta}(z) -2 - \theta(\theta-1)z^2\right| \leq 64\cdot z^3 \leq 64\varepsilon\cdot z^2.
\end{align*}
Again by the extended mean value theorem, it is written that
\begin{align*}
G_{\theta}'(z) = G_{\theta}'(0) + G_{\theta}''(0)z+\frac{G_{\theta}'''(d_{z})z^2}{2!},
\end{align*}
for some $d_{z} \in [0,z]$. Thus we gather that
\begin{align*}
|G_{\theta}'(z)| \leq 2\theta|\theta-1|z + 192\cdot z^2 \leq (2\theta|\theta-1|+192\varepsilon)z.
\end{align*}
\end{proof}
\begin{lem}
\label{kthetabound}
$k_{\theta}(x) \ll x^{\frac{1}{2\theta}}$.
\end{lem}
\begin{proof}
Observe that
\begin{align*}
\frac{I_{\theta,x}(k_{\theta}(x))}{s_{\theta}^{\theta}(x)} = G_{\theta}(\frac{k_{\theta}(x)}{s_{\theta}(x)}).
\end{align*}
By Lemma \ref{kthetasize1} there exists a number $N_{\theta}$ so that if $x \geq N_{\theta}$ then 
\begin{align*}
\left|\frac{k_{\theta}(x)}{s_{\theta}(x)}\right| \leq \frac{\theta|\theta-1|}{128}.
\end{align*}
\begin{itemize}
\item If $\theta \in (0,1)$ with Lemma \ref{Glocbound} and $\varepsilon = \frac{\theta|\theta-1|}{128}$ we have
\begin{align*}
\frac{I_{\theta,x}(k_{\theta}(x))}{s_{\theta}^{\theta}(x)} \leq 2 + \frac{\theta(\theta - 1)}{2}\left(\frac{k_{\theta}(x)}{s_{\theta}(x)}\right)^2.
\end{align*}
Using $I_{\theta,x}(k_{\theta}(x)) = x = 2s_{\theta}(x)^{\theta} +E_{\theta}(x)$ we obtain
\begin{align*}
\frac{\theta(1-\theta)}{2}\left(\frac{k_{\theta}(x)}{s_{\theta}(x)}\right)^2 \leq \frac{-E_{\theta}(x)}{s_{\theta}(x)^{\theta}},
\end{align*}
which gives us $k_{\theta}(x) \ll (-E_{\theta}(x))^{1/2}s_{\theta}(x)^{1-\frac{\theta}{2}}$. With Lemma \ref{Errorsizelem} we have
\begin{align*}
k_{\theta}(x) \ll (x^{1-1/\theta})^{1/2}(x^{1/\theta})^{1-\frac{\theta}{2}} = x^{\frac{1}{2\theta}}.
\end{align*}
\item If $\theta \in (1,2)$, we mimic the above steps, except using the inequality
\begin{align*}
\frac{I_{\theta,x}(k_{\theta}(x))}{s_{\theta}^{\theta}(x)} \geq 2 + \frac{\theta(\theta - 1)}{2}\left(\frac{k_{\theta}(x)}{s_{\theta}(x)}\right)^2.
\end{align*}
\end{itemize}
\end{proof}
Let \[
l_{\theta}(x) := \begin{cases}
                    \lfloor k_{\theta}(x)\rfloor & \theta \in (0,1)\\
                    \lceil k_{\theta}(x)\rceil & \theta \in (1,2),
                 \end{cases}
\]
so that with Lemma \ref{incdeclem} we have
\begin{align*}
(s_{\theta}(x) + l_{\theta}(x))^{\theta} + (s_{\theta}(x) - l_{\theta}(x))^{\theta} 
&= I_{\theta,x}(l_{\theta}(x)) \\
&\geq I_{\theta,x}(k_{\theta}(x))=x.
\end{align*}
We compute
\begin{align*}
I_{\theta,x}(l_{\theta}(x))-I_{\theta,x}(k_{\theta}(x)) &= s_{\theta}(x)^{\theta}\left(G_{\theta}(\frac{l_{\theta}(x)}{s_{\theta}(x)})- G_{\theta}(\frac{k_{\theta}(x)}{s_{\theta}(x)})\right)\\
&\leq s_{\theta}(x)^{\theta}\cdot G_{\theta}'(m_{\theta}(x))\cdot\frac{1}{s_{\theta}(x)},
\end{align*}
noting that in the above step, we applied the mean value theorem to the difference $G_{\theta}(\frac{l_{\theta}(x)}{s_{\theta}(x)})- G_{\theta}(\frac{k_{\theta}(x)}{s_{\theta}(x)})$, and $m_{\theta}(x)$ is a value between the numbers $\frac{k_{\theta}(x)}{s_{\theta}(x)}$ and $\frac{l_{\theta}(x)}{s_{\theta}(x)}$. With Lemmas \ref{Glocbound} and \ref{kthetasize1} it is noted that
\begin{align*}
I_{\theta,x}(l_{\theta}(x))-I_{\theta,x}(k_{\theta}(x)) \ll s_{\theta}(x)^{\theta}\cdot \frac{k_{\theta}(x)}{s_{\theta}(x)}\cdot\frac{1}{s_{\theta}(x)}.
\end{align*}
With $x^{1/\theta} \ll s_{\theta}(x) \ll x^{1/\theta}$ and Lemma \ref{kthetabound} we evaluate that
\begin{align*}
I_{\theta,x}(l_{\theta}(x))-I_{\theta,x}(k_{\theta}(x)) \ll (x^{1/\theta})^{\theta-2}x^{\frac{1}{2\theta}} = x^{1-\frac{3}{2\theta}}.
\end{align*}
Hence for $x$ sufficiently large, we have
\begin{align*}
x \leq (s_{\theta}(x) + l_{\theta}(x))^{\theta} + (s_{\theta}(x) - l_{\theta}(x))^{\theta} = x + O(x^{1-\frac{3}{2\theta}}),
\end{align*}
which proves the proposition for $\theta \in (0,1)\cup(1,2)$.

\subsection{The case $\theta \geq 2$} \text{}

\text{ }

\hspace{-0.4cm}Let $x\geq 2$, $f_{-,\theta}(x) = \lfloor x^{1/\theta} \rfloor^{\theta}$ and $f_{+,\theta}(x) = \lceil x^{1/\theta} \rceil^{\theta}$. Observe that we have the inequalities $f_{-,\theta}(x) \leq x$ and $f_{+,\theta}(x) \geq x$. Note that by the mean value theorem, we have
\begin{align*}
x - f_{-,\theta}(x) = (x^{1/\theta})^{\theta}-\lfloor x^{1/\theta} \rfloor^{\theta} \leq \theta\cdot x^{1-\frac{1}{\theta}}.
\end{align*}
With $y = x - f_{-,\theta}(x),$ and the mean value theorem we have
\begin{align*}
f_{+,\theta}(y) - y \leq \theta\cdot y^{1-\frac{1}{\theta}}.
\end{align*}
Hence we have 
\begin{align*}
x \leq \left\lceil\left(x-f_{-,\theta}(x)\right)^{1/\theta}\right\rceil^{\theta} + \left\lfloor x^{1/\theta}\right\rfloor^{\theta} &\leq x + \theta\cdot\left(\theta\cdot x^{1-\frac{1}{\theta}}\right)^{1-\frac{1}{\theta}}\\
&= x+\theta^{2-\frac{1}{\theta}}\cdot x^{1-\frac{2}{\theta}+\frac{1}{\theta^2}}.
\end{align*}
\subsection{Proof of Corollary \ref{cortwotheta}} \text{ }

\text{ }

\hspace{-0.4cm} Let $n \geq 5^{1/\theta}$ and $R(\theta) = \{j^{\theta}: \ j\in \mathbb{N}\}$. We can find an element $\alpha^{*} \in [\frac{n^{\theta}}{5}, \frac{2n^{\theta}}{5}]$ so that $\{\alpha^{*}\} = \{\alpha\}$. Due to Theorem \ref{main5} we can find an element $\beta_{n} \in \BigSquare{\theta}$ so that 
\begin{align*}
0\leq \beta_{n}-\alpha^{*}<D_{\theta}\cdot (\alpha^{*})^{1-\frac{3}{2\theta}} &\leq D_{\theta}\cdot 5^{\frac{3}{2\theta}-1}n^{\theta-3/2}\\
&= C_{\theta}\cdot n^{\theta-3/2}.
\end{align*}
\begin{itemize}
\item If $\beta_{n} \in R(\theta)$ then $\beta_{n} = j_{n}^{\theta}$ for some $j_{n} \in \mathbb{N}$. For sufficiently large $n$ we must have $j_{n} \leq n$. Observe that
\begin{align*}
C_{\theta}\cdot n^{\theta-3/2}\geq \|j_{n}^{\theta}-\alpha^{*}\| = \|j_{n}^{\theta}+1^{\theta}-\alpha^{*}\| = \|j_{n}^{\theta}+1^{\theta}-\alpha\|.
\end{align*}
\item If $\beta_{n}\not\in R(\theta)$ there exists positive integers $k_{n}$ and $l_{n}$ so that $\beta_{n}= k_{n}^{\theta}+l_{n}^{\theta}$. For sufficiently large $n$ we must have $k_{n},l_{n} \leq n$. Observe that
\begin{align*}
C_{\theta}\cdot n^{\theta-3/2}\geq \|k_{n}^{\theta}+l_{n}^{\theta}-\alpha^{*}\| = \|k_{n}^{\theta}+l_{n}^{\theta}-\alpha\|.
\end{align*}
\end{itemize}
\section{Proof of Theorem \ref{almost}}
Without loss of generality assume that $0<\rho(\cdot) \leq 1$ and let $0<\gamma < \Upsilon$ be real numbers. Define
\begin{align*}
P(m):= \left\{(a_{1},\ldots,a_{k}): a_{j}\in \mathbb{N}, \ 1 \leq a_{1}\ldots \leq a_{k} = m\right\},
\end{align*}
and put $f_{\theta} : \mathbb{N}^{k} \rightarrow \R$ with
\begin{align*}
f_{\theta}(a_{1},\ldots,a_{k}) = \sum_{j=1}^{k}a_{j}^{\theta}.
\end{align*}
\begin{lem}
\label{fbound2}
If $m \geq C_{\gamma}$ is a sufficiently large integer, $\beta \in [\gamma, \Upsilon]$ and $\omega \in P_{k}(m)$ then
\begin{align*}
f_{\beta - \frac{1}{\log(m)}}(\omega) \leq f_{\beta}(\omega) -1
\end{align*}
and
\begin{align*}
f_{\beta + \frac{1}{\log(m)}}(\omega) \geq f_{\beta}(\omega) +1.
\end{align*}
\end{lem}
\begin{proof}
Note that 
\begin{align*}
f_{\beta - \frac{1}{\log(m)}}(\omega) &\leq f_{\beta}(\omega) -m^{\beta}+m^{\beta - \frac{1}{\log(m)}}\\
&= f_{\beta}(\omega)-m^{\beta}+m^{\beta}\cdot e^{-1}\\
&\leq f_{\beta}(\omega)-m^{\gamma}+m^{\gamma}\cdot e^{-1},
\end{align*}
which is less than $f_{\beta}(\omega)-1$ for sufficiently large $m$. The second inequality is derived similarly.
\end{proof}
For $\omega \in P(m)$ we let 
\begin{align*}
D(\omega) = \left\{\theta: \ \theta \in [\gamma, \Upsilon], \  f_{\theta}(\omega) \in \mathbb{Z}\right\}.
\end{align*}
and $g(\omega) = \# D(w)$.
\begin{lem}
\label{Integerpointestimate}
If $m \geq 2$ and $\omega \in P(m)$ then $g(\omega) \leq k\cdot m^{\Upsilon}$.
\end{lem}
\begin{proof}
As a function of $\theta \in \mathbb{R}$, $f_{\theta}(\omega)$ is positive, furthermore it is strictly increasing, hence
\begin{align*}
g(\omega) &\leq \#\left\{\theta: \ \theta \in (-\infty, \Upsilon], \  f_{\theta}(\omega) \in \mathbb{Z}\right\}\\
&= \#\left\{\theta: \ \theta \in (-\infty, \Upsilon], \  f_{\theta}(\omega) \in 1,\ldots, \lfloor f_{\Upsilon}(\omega)\rfloor\right\}\\
&\leq f_{\Upsilon}(\omega)\\
&\leq k\cdot m^{\Upsilon}.
\end{align*}
\par \vspace{-\baselineskip} \qedhere
\end{proof}
For $m \geq 2$ we arrange the elements of $D(\omega)$ as
\begin{align*}
d(1,\omega)<\ldots<d(g(\omega),\omega).
\end{align*}
\begin{lem}
\label{dlowebound}
If $h \in \mathbb{N}$, $m \geq 2$ and $\omega \in P(m)$ then 
\begin{align*}
d(h,\omega) \geq \log_{m}\left(\frac{m^{\gamma}+h-1}{k}\right).
\end{align*}
\end{lem}
\begin{proof}
Observe that 
\begin{align*}
k\cdot m^{d(h,\omega)} &\geq f_{d(h,\omega)}(\omega)\\
&= f_{d(1,\omega)}(\omega)+h-1\\
&\geq m^{\gamma}+h-1.
\end{align*}
\par \vspace{-\baselineskip} \qedhere
\end{proof}
Using Lemma \ref{fbound2}, for $h \in \mathbb{N}$, $m \geq C_{\gamma}$ and $\omega \in P(m)$ we can define quantities $u_{+}(h,\omega)$ and $u_{-}(h,\omega)$ so that
\begin{align*}
f_{d(h,\omega)+u_{+}(h,\omega)}(\omega) = f_{d(h,\omega)}(\omega)+m^{-k}\cdot \rho(m)
\end{align*}
and
\begin{align*}
f_{d(h,\omega)-u_{-}(h,\omega)}(\omega) = f_{d(h,\omega)}(\omega)-m^{-k}\cdot \rho(m).
\end{align*}
Furthermore by Lemma \ref{fbound2} we have the inequalities
\begin{align}
\label{upmbasicbound}
0<u_{\pm}(h,\omega)\leq \frac{1}{\log(m)}.
\end{align}
We can improve these inequalities.
\begin{lem}
\label{upmenhancedbound}
For $h\in \mathbb{N}$, $m \geq C_{\gamma}$ and $\omega \in P(m)$ we have
\begin{align*}
0<u_{\pm}(h,\omega) \leq \frac{m^{-k-d(h,\omega)}\cdot \rho(m)}{(3-e)\log(m)}.
\end{align*}
\end{lem}
\begin{proof}
Note that
\begin{align*}
\left|m^{d(h,\omega)\pm u_{\pm}(h,\omega)} - m^{d(h,\omega)}\right| &\leq \left|f_{d(h,\omega) \pm u_{\pm}(h,\omega)}(\omega) - f_{d(h,\omega)}(\omega)\right|\\
&= m^{-k}\cdot\rho(m).
\end{align*}
Implying that
\begin{align}
\label{upmMedium}
\left|e^{\pm\log(m)u_{\pm}(h,\omega)}-1\right|\leq m^{-k-d(h,\omega)}\cdot\rho(m).
\end{align}
 For $|x| \leq 1$ note the inequality 
\begin{align}
\label{Einequality}
|e^{x}-1| \geq |x| - \sum_{j=2}^{\infty} \frac{|x|^{j}}{j!} \geq |x| - \sum_{j=2}^{\infty} \frac{|x|}{j!} = (3-e)|x|.
\end{align}
Using \eqref{upmbasicbound}, \eqref{upmMedium} and \eqref{Einequality} we obtain
\begin{align*}
(3-e)\cdot \left|\log(m)\cdot u_{\pm}(h,\omega) \right| \leq m^{-k-d(h,\omega)}\cdot \rho(m),
\end{align*}
and the result follows.
\end{proof}
Put
\begin{align*}
U(\omega) = \{\theta: \ \theta \in [\gamma, \Upsilon], \ \|f_{\theta}(\omega)\| \leq m^{-k}\cdot\rho(m) \}.
\end{align*}
With $\mu$ denoting the Lebesgue measure function we write the following lemma.
\begin{lem}
\label{Measureinequality1}
Let $m \geq C_{\gamma}$, $\omega \in P(m)$ we have the bound
\begin{align*}
\mu(U(\omega)) \leq B_{k,\Upsilon}\cdot m^{-k}\cdot \rho(m),
\end{align*}
for some constant $B_{k,\Upsilon}>0$.
\end{lem}
\begin{proof}
It is noted that
\begin{align*}
U(\omega) \subseteq \bigcup_{h=1}^{g(\omega)}\left[d(h,\omega)-u_{-}(h,\omega),d(h,\omega)+u_{+}(h,\omega) \right]
\end{align*}
Hence with Lemma \ref{upmenhancedbound} we have the inequality
\begin{align*}
\mu(U(\omega)) &\leq \sum_{h=1}^{g(\omega)}\frac{2\cdot m^{-k-d(h,\omega)}\cdot \rho(m)}{(3-e)\log(m)} \\
&=\frac{2\cdot m^{-k}\cdot \rho(m)}{(3-e)\log(m)}\sum_{h=1}^{g(\omega)}m^{-d(h,\omega)}.
\end{align*}
Using Lemma \ref{dlowebound} we continue the above inequality as
\begin{align*}
\mu(U(\omega)) &\leq \frac{2\cdot m^{-k}\cdot \rho(m)}{(3-e)\log(m)}\sum_{h=1}^{g(\omega)}m^{-\log_{m}\left(\frac{m^{\gamma}+h-1}{k}\right)}\\
&\leq \frac{2\cdot m^{-k}\cdot \rho(m)}{(3-e)\log(m)}\sum_{h=1}^{g(\omega)}\frac{k}{m^{\gamma}+h-1}.
\end{align*}
Using Lemma \ref{Integerpointestimate} we gather the estimate
\begin{align*}
\mu(U(\omega)) &\leq \frac{2\cdot m^{-k}\cdot \rho(m)}{(3-e)\log(m)}\sum_{h=1}^{\lceil k \cdot 
 m^{\Upsilon}\rceil}\frac{k}{h}\\
&\leq B_{k,\Upsilon}\cdot m^{-k}\cdot \rho(m), 
\end{align*}
for some constant $B_{k,\Upsilon}>0$.
\end{proof}
Let
\begin{align*}
V(m) &:= \left\{\theta :\ \theta \in [\gamma, \Upsilon], \ \exists \omega \in P(m), \ \|f_{\theta}(\omega)\| \leq m^{-k}\cdot \rho(m)\right\}\\
&\hspace{0.1cm}= \bigcup_{\omega \in P(m)}U(\omega).
\end{align*}
Hence with Lemma \ref{Measureinequality1}, when $m \geq C_{\gamma}$ we have the inequality
\begin{align}
\label{Vmmeasbound}
\mu(V(m)) &\leq \sum_{\omega \in P(m)}\mu(U(\omega)) \notag\\
&\leq m^{k-1}\cdot B_{k,\Upsilon}\cdot m^{-k}\cdot \rho(m) \notag\\
&=B_{k,\Upsilon}\cdot m^{-1}\cdot \rho(m).
\end{align}
Let
\begin{align*}
W &= \left\{
\theta \colon \theta \in [\gamma, \Upsilon], \ 
\exists m_1 < m_2 < \dots \in \mathbb{N}, \ 
\exists \omega_j \in P(m_j), \ 
\right. \\
  &\hspace{2.85cm}\left. \| f_{\theta}(\omega_j) \| \leq m_j^{-k} \cdot \rho(m_j) 
\right\} \\
  &= \bigcap_{n=1}^{\infty} \bigcup_{m=n}^{\infty} V(m).
\end{align*}
With inequality \eqref{Vmmeasbound} and $\sum_{m=1}^{\infty}m^{-1}\cdot \rho(m) < \infty$, it is deduced that $\mu(W) = 0$, due to Borel-Cantelli lemma for measure spaces. Hence for almost all $\theta \in [\gamma, \Upsilon]$ the inequality $\|f_{\theta}(\omega)\| \leq \|\omega\|_{\infty}^{-k}\cdot \rho(\|\omega\|_{\infty})$ has finitely many solutions for $\omega \in \mathbb{N}^{k}$. Theorem \ref{almost} follows as $0<\gamma < \Upsilon$ are arbitrary.

\section{Proof of Theorem \ref{DubCounter}}
It is enough to show that when $r,s \in \mathbb{N}$ and $s \geq 2$, there are uncountably many $\theta \in [r/s, r/(s-1)]$ which satisfy the property that the inequality $0<\|n^{\theta}\| \leq \Phi(n)$ holds for infinitely many integers. Put
\begin{align*}
S = \{(s_{j})_{j=0}^{\infty}: s_{j} \in \N, \ s_{0}=s, \ s_{n+1}\geq s_{n} \},
\end{align*}
so that $S$ is a subset of integer-valued sequences. Define $\phi : S\rightarrow \R$ so that 
\begin{align*}
\phi((s_{j})_{j=0}^{\infty}) = \frac{r}{s}+\sum_{d=1}^{\infty}\frac{1}{\prod_{k=0}^{d}s_{k}}.
\end{align*}
Observe that $\phi$ is injective on $S$ and lies in the interval $[r/s, r/(s-1)]$. With $\mathbf{s} = (s_{j})_{j=0}^{\infty}$, we see that for each $h \in \mathbb{N}\cup\{0\}$ there exists a positive number $U(\mathbf{s},h)$ so that
\begin{align*}
\phi(\mathbf{s}) = \frac{U(\mathbf{s},h)}{\prod_{k=0}^{h}s_{k}}+\sum_{d=h+1}^{\infty}\frac{1}{\prod_{k=0}^{d}s_{k}}.
\end{align*}
Define the set
\begin{align*}
S(\Phi) &:=\left\{\mathbf{s}: \mathbf{s}\in S, \ \forall h \in \mathbb{N}\cup\{0\} \text{ we have }\right.\\
\hspace{2cm}&\left. \left|2^{U(\mathbf{s},h)+\frac{1}{s_{h+1}-1}}-2^{U(\mathbf{s},h)}\right|\leq \Phi\left(2^{\prod_{k=0}^{h}s_{k}}\right)\right\},
\end{align*}
and put
\begin{align*}
S^{*}(\Phi):= \left\{\phi(\mathbf{s}):\ \mathbf{s}\in S(\Phi)\right\}\setminus\{x: \ x>0, \ \exists n \in \N_{\geq 2} \text{ so that }\|n^{x}\|=0\}.
\end{align*}
It is noted that $S^{*}(\Phi)$ has uncountably many elements in the interval $[r/s, r/(s-1)]$, furthermore if $\theta \in S^{*}(\Phi)$ has the representation
\begin{align*}
\theta = \frac{r}{s}+\sum_{d=1}^{\infty}\frac{1}{\prod_{k=0}^{d}s_{k}},
\end{align*}
then 
\begin{align*}
\left|(2^{\prod_{k=0}^{h}s_{k}})^{\theta} - 2^{U(\mathbf{s},h)}\right| &= \left| 2^{U(\mathbf{s},h)+\sum_{d=h+1}^{\infty}\frac{1}{\prod_{k=h+1}^{d}s_{k}}}-2^{U(\mathbf{s},h)} \right|\\
&\leq \left|2^{U(\mathbf{s},h)+\frac{1}{s_{h+1}-1}}-2^{U(\mathbf{s},h)}\right| \leq \Phi\left(2^{\prod_{k=0}^{h}s_{k}}\right).
\end{align*}
Thus with such $\theta$ we have the inequality
\begin{align*}
0<\left\|\left(2^{\prod_{k=0}^{h}s_{k}}\right)^{\theta}\right\| \leq \Phi\left(2^{\prod_{k=0}^{h}s_{k}}\right)
\end{align*}
for $h=0,1,\ldots$.
\section{Acknowledgements}

\begin{itemize}
\item The author thanks Bryce Kerr for his thoughtful review of an earlier version of this manuscript. Initially, when the author considered the case $d=2$ in Theorem \ref{main}, Bryce's insightful suggestion to select $x_j$ with $M(x_j) \leq 2^j$ (rather than $M(x_{j}) \leq j$ as originally thought of by the author) in Lemma \ref{AlgorithmicIdea} directly improved the exponent in Theorem \ref{main} (for $d=2$) from $\frac{-k}{8}+\frac{1}{8}$ to $\frac{-k}{4}+\frac{1}{4}$.

\item The author expresses deep gratitude to Igor Shparlinski for recommending Steinerberger's paper \cite{Steinerberger} and for invaluable insights.

\item The author thanks Stefan Steinerberger and Art{\=u}ras Dubickas for their continued support and encouragement.
 
\item Finally, the author gratefully acknowledges the support of the Australian Government Research Training Program (RTP) Scholarship and the University of New South Wales for a top-up scholarship, both of which were instrumental in this work.
\end{itemize}

\end{document}